\documentclass[11pt]{article}

\usepackage[utf8]{inputenc} 
\usepackage{amsmath}
\usepackage{graphicx}
\usepackage{subcaption}
\usepackage[top=30mm, bottom=30mm, left=27mm, right=27mm]{geometry}
\usepackage{amsfonts}
\usepackage{mathtools}
\usepackage{subcaption}
\usepackage{amssymb}
\usepackage{amsthm}
\usepackage{chngpage}
\usepackage{hyperref}
\hypersetup{
	colorlinks=true,
	linkcolor=black,
	citecolor=black,
	filecolor=black,      
	urlcolor=black,
}
\usepackage{enumerate}
\usepackage{makeidx}
\usepackage{bm}
\usepackage{thmtools}
\usepackage{thm-restate}

\makeindex

\makeatletter
\renewenvironment{proof}[1][\proofname] {\par\pushQED{\qed}\normalfont\topsep6\p@\@plus6\p@\relax\trivlist\item[\hskip\labelsep\bfseries#1\@addpunct{.}]\ignorespaces}{\popQED\endtrivlist\@endpefalse}
\makeatother

\newcommand{\ex}{\operatorname{ex}}
\newcommand{\ZZ}{\mathbb{Z}}

\newtheorem{proposition}{Proposition}[section]
\newtheorem{lemma}[proposition]{Lemma}
\newtheorem{theorem}[proposition]{Theorem}

\usepackage{amsthm}

\theoremstyle{definition}

\newtheorem*{remark*}{Remark}
\newtheorem*{theorem*}{Theorem}

\mathcode`l="8000
\begingroup
\makeatletter
\lccode`\~=`\l
\DeclareMathSymbol{\lsb@l}{\mathalpha}{letters}{`l}
\lowercase{\gdef~{\ifnum\the\mathgroup=\m@ne \ell \else \lsb@l \fi}}%
\endgroup

\title{The generalised rainbow Turán problem for cycles}
\author{Barnabás Janzer\thanks{Department of Pure Mathematics and Mathematical Statistics, Centre for Mathematical Sciences, University of Cambridge, Wilberforce Road, Cambridge CB3 0WB, United Kingdom. Email: bkj21@cam.ac.uk}}
\date{\vspace{-21pt}}

\begin{document}
	\maketitle

\begin{abstract}
Given an edge-coloured graph, we say that a subgraph is rainbow if all of its edges have different colours. 
Let $\ex(n,H,$ rainbow-$F)$ denote the maximal number of copies of $H$ that a properly edge-coloured graph on $n$ vertices can contain if it has no rainbow subgraph isomorphic to $F$. We determine the order of magnitude of $\ex(n,C_s,$ rainbow-$C_t)$ for all $s,t$ with $s\not =3$. In particular, we answer a question of Gerbner, Mészáros, Methuku and Palmer by showing that $\ex(n,C_{2k},$ rainbow-$C_{2k})$ is $\Theta(n^{k-1})$ if $k\geq 3$ and $\Theta(n^2)$ if $k=2$. We also determine the order of magnitude of $\ex(n,P_\ell,$ rainbow-$C_{2k})$ for all $k,\ell\geq 2$, where $P_\ell$ denotes the path with $\ell$ edges.

\end{abstract}\vspace{-5pt}	

\section{Introduction}

The problem of estimating the maximal possible size $\ex(n,F)$ of an $F$-free graph on $n$ vertices is one of the most fundamental problems in extremal graph theory. It is a well known fact that $\ex(n,F)/\binom{n}{2}\to 1-1/(r-1)$ as $n\to \infty$ if $F$ has chromatic number $r$, determining the asymptotic behaviour of this function when $F$ is not bipartite. However, much less is known in the bipartite case. See \cite{furedi2013history} for a survey on the topic.

Alon and Shikhelman introduced \cite{alon2016many} the following generalisation of the problem above. Given two graphs $H$ and $F$, let $\ex(n,H,F)$ denote the maximal number of copies of $H$ that an $F$-free graph on $n$ vertices can contain. Note that the usual Turán number $\ex(n,F)$ is the special case $\ex(n,K_2,F)$. This problem has been studied for several different choices of $H$ and $F$, see e.g. \cite{alon2016many,gerbner2017generalized,gishboliner2018generalized}.

Another generalisation of the Turán problem was introduced by Keevash, Mubayi, Sudakov and Verstra\"ete \cite{keevash2007rainbow}. Given an edge-coloured graph, we say that a subgraph is rainbow if all of its edges have different colours. Let $\ex^*(n,F)$ denote the maximal numer of edges that a properly edge-coloured graph on $n$ vertices can have if it contains no rainbow copy of $F$. Note that clearly $\ex(n,F)\leq \ex^*(n,F)$, and in fact $\ex^*(n,F)=\ex(n,F)+o(n^2)$, giving the asymptotic behaviour when $F$ is not bipartite \cite{keevash2007rainbow}. This rainbow Turán problem has been studied for graphs $F$ including paths \cite{johnston2016rainbow,ergemlidze2019rainbow}, cycles \cite{keevash2007rainbow,das2013rainbow} and complete bipartite graphs \cite{keevash2007rainbow}, and for several graphs exact results are also known \cite{keevash2007rainbow}.

A common generalisation was studied by Gerbner, Mészáros, Methuku and Palmer \cite{gerbner2019generalized}. Let $\ex(n,H,\textnormal{rainbow-}F)$ denote the maximal number of copies of $H$ that a properly edge-coloured graph on $n$ vertices can contain if it has no rainbow subgraph isomorphic to $F$. The authors of \cite{gerbner2019generalized} focused mainly on the case $H=F$, and obtained several results, for example when $F$ is a path, cycle or a tree. Concerning cycles, they proved the following theorem.
\begin{theorem}[Gerbner, Mészáros, Methuku, Palmer \cite{gerbner2019generalized}]\label{theorem_GMMPcycle}
	If $k\geq 2$ is an integer, then
	\[\ex(n,C_{2k+1},\textnormal{rainbow-}C_{2k+1})=\Theta(n^{2k-1})\]
	and
	\[\Omega(n^{k-1})\leq \ex(n,C_{2k},\textnormal{rainbow-}C_{2k})\leq O(n^{k}).\]
	Moreover, if $l\geq 2$ is an integer with $l\not =k$, then
	 \[\ex(n,C_{2l},\textnormal{rainbow-}C_{2k})=\Theta(n^l).\] 
\end{theorem}
(Throughout this paper, whenever we use the $\Omega, \Theta$ or $O$ notation, the implied constants may depend, as usual, on the other parameters present, such as $k$ and $l$ above.) The authors of \cite{gerbner2019generalized} asked what the correct order of magnitude is for $\ex(n,C_{2k},\textnormal{rainbow-}C_{2k})$. (They were able to improve the lower bound to $\Omega(n^{3/2})$ when $k=2$ and the upper bound to $O(n^{8/3})$ when $k=3$.) The main aim of this paper is to obtain the following extension of Theorem \ref{theorem_GMMPcycle}.


\begin{theorem}\label{theorem_allcycles}
If $s\geq 4$ and $t\geq 3$ are positive integers, then
\begin{equation*}
\ex(n,C_s,\textnormal{rainbow-}C_t)=
\begin{cases*}
\Theta(n^{s/2}) & if $t=4$\\
\Theta(n^{s/2}) & if $s,t$ are even with $s\not =t$\\
\Theta(n^{s/2-1}) & if $s=t\geq 6$ and $t$ is even\\
\Theta(n^{(s-1)/2}) & if $t\geq 6$ is even and $s$ is odd\\
\Theta(n^{s-2}) & if $s,t$ are odd with $s\leq t$\\
\Theta(n^s) & if $t$ is odd, and $s>t$ or $s$ is even.
\end{cases*}
\end{equation*}
\end{theorem}

For comparison, we mention the order of magnitude of this function in the non-rainbow setting. 
We note that in many cases more precise bounds are known than the ones given below.

\begin{theorem}[Gishboliner, Shapira \cite{gishboliner2018generalized}, Gerbner, Győri, Methuku, Vizer \cite{gerbner2017generalized}]\label{theorem_nonrainbow}
	If $s\geq 4$ and $t\geq 3$ are distinct positive integers, then
	\begin{equation*}
	\ex(n,C_s,C_t)=
	\begin{cases*}
\Theta(n^{s/2}) & if $t=4$\\
\Theta(n^{s/2})	& if $s,t$ are even\\
\Theta(n^{(s-1)/2}) & if $t\geq 6$ is even and $s$ is odd\\
\Theta(n^{(s-1)/2}) & if $s,t$ are odd with $s<t$\\
\Theta(n^{s}) & if $t$ is odd, and $s>t$ or $s$ is even.
	\end{cases*}
	\end{equation*}
\end{theorem}

As part of our proof, we will also determine the order of magnitude of the maximal number of paths of length $l$ if there is no rainbow copy of $C_{2k}$ whenever $k,l\geq 2$. (By the path $P_l$ of length $l$ we mean the path with $l$ edges and $l+1$ vertices.) This result is given in the following theorem. Note that the answer is of the same order of magnitude as in the case of the corresponding (non-rainbow) generalised Turán problem \cite{gishboliner2018generalized}, although our proof is rather different. Also, we trivially have $\ex(n,P_l,\textnormal{rainbow-}C_{t})=\Theta(n^{l+1})$ if $t$ is odd.


%

\begin{theorem}\label{theorem_pathcycle}
 If $k,l\geq 2$ are integers, then
 \begin{equation*}
\ex(n,P_l,\textnormal{rainbow-}C_{2k})=
\begin{cases*}
\Theta(n^{\lceil (l+1)/2\rceil}) & if $k\geq 3$ \\
\Theta(n^{l/2+1}) & if $k=2$.\\
\end{cases*}
\end{equation*} 
\end{theorem}

Note that a path of length $l=1$ is just an edge, so the corresponding generalised rainbow Turán number $\ex(n,P_1,\textnormal{rainbow-}C_{2k})$ is $\ex^*(n,C_{2k})$. The correct order of magnitude of this is unknown for $k\geq 4$, but is conjectured to be $\Theta(n^{1+1/k})$ for all $k$ (and a corresponding lower bound is known) \cite{keevash2007rainbow,das2013rainbow}. We mention that we believe that the most difficult (new) results in this paper are Theorem \ref{theorem_pathcycle} and the closely related $s=t=2k$ case of Theorem \ref{theorem_allcycles}.\medskip


	
Theorem \ref{theorem_allcycles} deals with all cases except when $s=3$. In that case the correct order of magnitude is unknown in general even in the non-rainbow setting, where the following bounds are known. 

\begin{theorem}[Győri, Li \cite{gyHori2012maximum}, Alon, Shikhelman \cite{alon2016many}, Gishboliner, Shapira \cite{gishboliner2018generalized}]\label{theorem_trianglenonrainbow}
For every $k\geq 2$, we have
\[\Omega(\ex(n,\{C_4,C_6,\dots, C_{2k}\}))\leq \ex(n,C_3,C_{2k})\leq O(\ex(n,C_{2k}))\]
and 
\[\Omega(\ex(n,\{C_4,C_6,\dots, C_{2k}\}))\leq \ex(n,C_3,C_{2k+1})\leq O(\ex(n,C_{2k})).\]
\end{theorem}
	
Note that the lower and upper bounds are only known to be of the same order of magnitude when $k\in\{2, 3, 5\}$, in which case both bounds are $\Theta(n^{1+1/k})$. For the rainbow version, we have the following.
\begin{theorem}\label{proposition_triangle}
If $k\geq 2$ is odd then $\ex(n,C_3,\textnormal{rainbow-}C_{2k})=\Omega(n^{1+1/k})$, and if $k$ is even then $\ex(n,C_3,\textnormal{rainbow-}C_{2k+1})=\Omega(n^{1+1/k})$.
Furthermore, for every $k\geq 2$ integer, we have
\[\ex(n,C_3,\textnormal{rainbow-}C_{2k})=O(\ex^*(n,C_{2k}))\]
and 
\begin{align*}\ex(n,C_3,\textnormal{rainbow-}C_{2k})\geq \ex(n,C_3,C_{2k})=\Omega(\ex(n,\{C_4,C_6,\dots,C_{2k}\})),\\
\ex(n,C_3,\textnormal{rainbow-}C_{2k+1})\geq \ex(n,C_3,C_{2k+1})=\Omega(\ex(n,\{C_4,C_6,\dots,C_{2k}\})).
\end{align*}
\end{theorem}

Note that, as mentioned before, $\ex^*(n,C_{2k})$ is conjectured \cite{keevash2007rainbow,das2013rainbow} to be $O(n^{1+1/k})$, and $\ex(n,\{C_4,C_6,\dots,C_{2k}\})$ is only known to be $\Omega(n^{1+1/k})$ when $k=2, 3, 5$.\medskip


\section{Forbidden rainbow $C_{2k}$}\label{section_evencycle}
In this section we consider graphs having no rainbow $C_{2k}$ subgraph, and prove the corresponding cases of Theorem \ref{theorem_allcycles}, as well as Theorem \ref{theorem_pathcycle} concerning the number of paths. We will use the following lemma of Gerbner, Mészáros, Methuku and Palmer \cite{gerbner2019generalized}. We also include its proof below for completeness.

\begin{lemma}[Gerbner, Mészáros, Methuku, Palmer \cite{gerbner2019generalized}]\label{lemma_pathsoflength2}
	Let $G$ be a properly edge-coloured graph on $n$ vertices containing no rainbow $C_{2k}$. Then for every $a\in V(G)$, the number of paths $axy$ of length 2 starting at $a$ is $O(n)$.
\end{lemma}

\begin{proof}
	We may assume that $G$ is bipartite, since a random bipartition is expected to preserve a quarter of all paths of length $2$ starting at $a$. Let $X=N(a)$ and $Y=N(N(a))\setminus\{a\}$. Observe that the number of paths $axy$ is $e(X,Y)$, that is, the number of edges between $X$ and $Y$. So it suffices to show that the induced subgraph $G[X\cup Y]$ does not contain a $(100k)$-ary tree of depth $2k$.
	
	Assume that it does contain such a tree. Then it also contains a $(100k)$-ary tree of depth $2k-1$ rooted at some $x_1\in X$. Then we can recursively find distinct vertices $y_1,x_2,y_2,\dots,y_{k-1},x_k$ (with $x_i\in X, y_j\in Y$) such that for all $i$, $x_iy_i, y_ix_{i+1}\in E(G)$, and the colours $c(x_iy_i), c(y_ix_{i+1}), c(ax_i)$ are all distinct. (Here $c$ denotes the edge-colouring.) But then $ax_1y_1x_2y_2\dots y_{k-1}x_ka$ is a rainbow cycle of length $2k$, giving a contradiction.
\end{proof}

We now state explicitly the cases of Theorem \ref{theorem_allcycles} we deal with in the next two subsections.

\begin{theorem}\label{theorem_rainbowcycle}
	Let $k\geq 2$ be an integer. Then
	\begin{equation*}
	\ex(n,C_{2k},\textnormal{rainbow-}C_{2k})=
	\begin{cases*}
\Theta(n^{k-1}) & if $k\geq 3$ \\
\Theta(n^2) & if $k=2$.\\
	\end{cases*}
	\end{equation*}
\end{theorem}

\begin{theorem}\label{theorem_oddcycle}
	If $k, l\geq 2$ are integers, then
	\begin{equation*}\ex(n,C_{2l+1},\textnormal{rainbow-}C_{2k})=
	\begin{cases*}
		\Theta(n^{l}) & if $k\geq 3$ \\
		\Theta(n^{l+1/2}) & if $k=2$.\\
	\end{cases*}
	\end{equation*}
\end{theorem}

For the remainder of this section, unless otherwise stated, we will assume that $k\geq 2$ is an integer, $G$ is a properly edge-coloured graph on $n$ vertices with no rainbow copy of $C_{2k}$, and $c: E(G)\to \ZZ$ denotes the edge-colouring.

\subsection{Paths and even cycles}
In this subsection, we will prove Theorems \ref{theorem_pathcycle} and \ref{theorem_rainbowcycle}. Note that for the upper bounds in Theorems \ref{theorem_pathcycle} and \ref{theorem_rainbowcycle} it suffices to consider bipartite graphs $G$, since a random bipartition is expected to preserve a fixed positive proportion of subgraphs isomorphic to a given bipartite graph, so from now on we assume that $G$ is bipartite. 


In light of Lemma \ref{lemma_pathsoflength2}, to prove the upper bound in Theorem \ref{theorem_pathcycle} for $k\geq 3$, it is sufficient to show that the number of paths of length 3 is $O(n^2)$. Let us say that a pair $x,y$ of vertices of $G$ is bad if $x$ and $y$ have at least $100k$ common neighbours, and it is good otherwise. Then there are three types of paths $axyz$ of length $3$: either $ay$ and $xz$ are both good, or both bad, or one of them is good and the other one is bad. We will treat these cases in separate lemmas. It will be important later that for two of these cases we prove not only that the number of $P_3$s of that type is $O(n^2)$, but also that any vertex is a certain endpoint of $O(n)$ such $P_3$s. However, it is not true that for any vertex $a$ the number of paths $axyz$ of length $3$ starting at $a$ has to be $O(n)$. To see this, take a $C_{2k}$-free bipartite graph $G_0$ on vertex classes $X,Y$ with $|X|=|Y|=n/4$ and $|E(G_0)|=\omega(n)$. For each $x\in X$ add a new vertex $x'$, and join each pair $xx'$ by an edge of the same colour. Finally, add a vertex $a$ and join it to all vertices $x'$. Then the (bipartite) graph we get contains no rainbow $C_{2k}$, and the number of paths of length $3$ starting at $a$ is $|E(G_0)|$.
	


%
%

\begin{lemma}\label{lemma_badbadpath}
	Let $k\geq 3$. For every $a\in V(G)$, the number of paths $axyz$ such that $ay$ and $xz$ are both bad is $O(n)$.
\end{lemma}

\begin{proof}
	Let $Y=\{y\in N(N(a))\setminus\{a\}:ay\textnormal{ is bad}\}$, and let $Z=N(Y)$. Observe that $G[Y\cup Z]$ cannot contain a rainbow path of length $2k-3$. Indeed, if there is such a rainbow path, then there is a rainbow path $y_1z_1\dots y_{k-2}z_{k-2}y_{k-1}$ of length $2k-4$ with $y_i\in Y, z_j\in Z$. Since $ay_1$ and $ay_{k-1}$ are bad, we can choose $b\in N(a)\cap N(y_1)$ and $b'\in N(a)\cap N(y_{k-1})$ such that $aby_1z_1\dots y_{k-2}z_{k-2}y_{k-1}b'a$ is a rainbow $2k$-cycle, giving a contradiction. It follows that $e(Y,Z)=O(n)$, i.e., $\sum_{y\in Y}{\deg_G(y)}=O(n)$. (We are using the fact that for any $l$ we have $\ex^*(n,P_l)=O(n)$. See \cite{ergemlidze2019rainbow} for the best known upper bound.)
	
	For each $y\in Y$, define an auxiliary graph $H_y$ on vertex set $N(y)$ by letting $zz'$ be an edge if and only if $zz'$ is bad. Note that $H_y$ cannot contain a path of length $k-1$. Indeed, if $z_1\dots z_k$ is such a path, then we can choose $b_i\in N_G(z_i)\cap N_G(z_{i+1})$ in such a way that $az_1b_1z_2\dots b_{k-1}z_ka$ is a rainbow $2k$-cycle in $G$, giving a contradiction. It follows that $|E(H_y)|\leq k|H_y|=k\deg_G(y)$. But the number of triples $(x,y,z)$ such that $xyz$ is a path, $xz$ is bad and $y\in Y$ is $2\sum_{y\in Y}|E(H_y)|\leq 2k\sum_{y\in Y}\deg_G(y)=O(n)$. The statement of the lemma follows.
\end{proof}

\begin{lemma}\label{lemma_goodbadpath}
	Let $k\geq 3$. For every $a\in V(G)$, the number of paths $axyz$ such that $ay$ is good and $xz$ is bad is $O(n)$.
\end{lemma}

\begin{proof}
	Let $Y=\{y\in N(N(A))\setminus\{a\}: ay\textnormal{ is good}\}$, and let
	\begin{align*}Z=\{z\in V(G): \textnormal{for any set $S\subseteq V(G)$ with $|S|\leq 100k$ there is a path $axyz$ of length 3}\\ \textnormal{such that $x\not\in S$, $ay$ is good and $xz$ is bad}\}.\end{align*}

	Consider first the number of paths $axyz$ with $z\in Z$ such that $ay$ is good (and $xz$ is bad). The number of these is at most $100k\cdot e(Y,Z)$, as after picking $yz$ there are at most $100k$ possible choices for $x$. \vspace{-3pt}
	\subparagraph{Claim.} $G[Y\cup Z]$ cannot contain a rainbow path of length $2k-5$.\vspace{-3pt} \subparagraph{Proof of Claim.} Suppose it contains such a rainbow path. Then it also contains a rainbow path $P$: $z_1y_1\dots z_{k-3}y_{k-3}z_{k-2}$ of length $2k-6$ such that $z_i\in Z, y_j\in Y$. Let
	\[S_1=V(P)\cup\{x\in N(a): c(ax)=c(z_iy_i)\textnormal{ or }c(ax)=c(y_iz_{i+1})\textnormal{ for some $i$}\}.\]
	Then $|S_1|<100k$, so we can pick a $P_3$ $axyz_1$ from $a$ to $z_1$ such that $x\not \in S_1$, $ay$ is good and $xz_1$ is bad. Let $S_2=S_1\cup \{x\}$ and pick a path $ax'y'z_{k-2}$ such that $x'\not \in S_2$, $ay'$ is good and $x'z_{k-2}$ is bad. Then we can pick $y''\in N(x)\cap N(z_1)$ such that $c(xy'')$ and $c(y''z_1)$ are distinct from all $c(z_iy_i), c(y_iz_{i+1}), c(ax),c(ax')$, and $y''$ is distinct from $a$ and each $y_i$. Similarly, we can pick $y'''$ such that $c(xy''')$ and $c(y'''z_1)$ are distinct from all $c(z_iy_i), c(y_iz_{i+1}), c(ax),c(ax'), c(xy''),c(y''z_1)$, and $y'''$ is distinct from $a, y''$ and each $y_i$. Then $axy''z_1y_1z_2\dots y_{k-3}z_{k-2}y'''x'a$ is a rainbow $C_{2k}$, giving a contradiction. The claim follows.\medskip
	
	So $G[Y\cup Z]$ contains no rainbow $P_{2k-5}$, so $e(Y,Z)=O(n)$. So there are $O(n)$ $P_3s$ $axyz$ with $z\in Z$ such that $ay$ is good (and $xz$ is bad).\medskip
	
	Now consider the number of $P_3s$ $axyz$ with $z\not \in Z$ such that $ay$ is good and $xz$ is bad. Given $z\not \in Z$, there is a set $S$ with $|S|\leq 100k$ such that any $P_3$ $axyz$ such that $ay$ is good and $xz$ is bad must have $x\in S$. So for each $z\in Z$ we can pick $x_z\in N(a)$ such that at least a proportion of $1/(100k)$ of all such $P_3$s from $a$ to $z$ go through $x_z$. For each $x\in N(a)$ let $Z_x=\{z\not\in Z: x_z=x\}$. Also let $Y_x=Y\cap N(x)$. Then the number of such $P_3$s starting in $a$ and ending outside $Z$ is at most
	\begin{align*}
	\sum_{z\not \in Z} 100k\cdot  |N(x_z)\cap N(z)\cap Y|&=\sum_{z\not \in Z} 100k\cdot  e(Y_{x_z},\{z\})\\
	&=\sum_{x \in N(a)} 100k\cdot e(Y_x,Z_x).
	\end{align*}
	
	Note that $e(Y_x,Z_x)$ is the number of paths of length $2$ starting at $x$ in the graph $G[\{x\}\cup Y_x\cup Z_x]$. Since that graph contains no rainbow $C_{2k}$, Lemma \ref{lemma_pathsoflength2} gives that $e(Y_x,Z_x)=O(|Y_x|+|Z_x|+1)$. Note, however, that
	\[\sum_{x\in N(a)} |Y_x|=\sum_{y\in Y}|N(y)\cap N(a)|\leq 100k|Y|=O(n)\]
	and
	\[\sum_{x\in N(a)} |Z_x|=\sum_{z\not \in Z}1=O(n).\]
	
	Putting together, we get that the number of such $P_3$s starting at $a$ and ending outside $Z$ is $O(n)$. The statement of the lemma follows.
\end{proof}

\begin{lemma}\label{lemma_goodgoodpath}
	Let $k\geq 3$. The number of paths $axyz$ such that $ay$ and $xz$ are both good is $O(n^2)$.
\end{lemma}

Some parts of the proof below will be similar to the proof of the fact $\ex^*(n,C_6)=O(n^{4/3})$ in \cite{keevash2007rainbow}.

\begin{proof}
We start similarly as in the proof of Lemma \ref{lemma_goodbadpath}. Let
\begin{align*}
W=\{(a,z)\in V(G)\times V(G):\textnormal{ for any set $S$ of at most $(100k)^2$ colours there is a rainbow path $axyz$}\\\textnormal{such that $c(ax),c(xy),c(yz)\not\in S$ and $ay,xz$ are good.}\}
\end{align*}
Given $a\in V(G)$, let $Z_a=\{z: (a,z)\in W\}$, and let $Y_a=\{y\in N(N(a))\setminus\{a\}:\textnormal{ $ay$ is good}\}$.

\subparagraph{Claim.} $G[Y_a,Z_a]$ contains no rainbow path of length $2k-5$.\vspace{-3pt} \subparagraph{Proof of Claim.} Suppose it does. Then it also contains a rainbow path $P: z_1y_1\dots z_{k-3}y_{k-3}z_{k-2}$ of length $2k-6$ with $z_i\in Z_a, y_j\in Y_a$. Let
\[S_1=\bigcup_i\{c(z_iy_i),c(y_iz_{i+1})\}\cup\bigcup_i \{c(ax): x\in N(a)\cap N(y_i)\}\cup\bigcup_i \{c(xy_i): x\in N(a)\cap N(y_i)\}\cup\{c(az_i):z_i\in N(a)\}.\]
Note that $|S_1|\leq 2k+ k\cdot  100k+k\cdot 100k+k<(100k)^2$, so we can pick a rainbow path $axyz_1$ such that $ay$, $xz_1$ are good and $c(ax), c(xy), c(yz_1)\not \in S_1$. Note that $y\not =y_i$ for all $i$ and $x\not =z_j$ for all $j$. Let 
\[S_2=S_1\cup\{c(ax),c(xy),c(yz_1)\}\cup \{c(aw):w\in N(a)\cap N(y)\}\cup \{c(wy):w\in N(a)\cap N(y)\}.\]
We have $|S_2|<(100k)^2$, so we can pick a rainbow path $ax'y'z_{k-2}$ such that $ay',x'z_{k-2}$ are good and $c(ax'), c(x'y'), c(x'z_{k-2})\not \in S_2$. Note that $y'\not =y_i,y$ and $x'\not =z_j,x$. But then $axyz_1y_1\dots z_{k-3}y_{k-3}z_{k-2}y'x'a$ is a rainbow $C_{2k}$, giving a contradiction. The claim follows.\medskip

By the Claim, we have $e(Y_a,Z_a)=O(n)$ for all $a$. Hence the number of paths $axyz$ such that $ay$ and $xz$ are good and $(a,z)\in W$ is $O(n^2)$ (since for any $a$, each edge $yz$ extends to at most $100k$ such paths $axyz$).

Now consider $P_3$s $axyz$ with $(a,z)\not \in W$. For any $a$ and $z$, let $f(a,z)$ denote the number of rainbow $P_3$s $axyz$ from $a$ to $z$ such that $ay$ and $xz$ are both good. If $(a,z)\not \in W$, we can pick a colour $c_{az}$ such that there are at least $\lceil f(a,z)/(100k)^2\rceil$ $P_3$s $axyz$ such that $ay, xz$ are good and $c_{az}\in \{c(ax),c(xy),c(yz)\}$. Note that at most $100k$ of these $P_3$s have $c(ax)=c_{az}$, since the colouring is proper and $xz$ is good. Similarly, at most $100k$ of these $P_3$s have $c(yz)=c_{az}$. We deduce that there are at least $N_{az}=\lceil f(a,z)/(100k)^2\rceil -200k$ $P_3$s $axyz$ such that $c(xy)=c_{az}$ and $ay, xz$ are good. Note that these paths must be internally vertex-disjoint. So we can list $N_{az}$ such paths as $ax_iy_iz$ for $i=1,2,\dots,N_{az}$ such that if $i\not =j$ then $x_i\not =x_j$ and $y_i\not =y_j$.

Using the observations above, we now show that there are `many' 6-cycles $ax_iy_izy_jx_ja$ such that $c(x_iy_i)=c(x_jy_j)=c_{az}$ and each pair (of distance 2) in the 6-cycle is good. (Note that if we did not require that $xx'$ and $yy'$ are good then we would immediately get at least $\binom{N_{az}}{2}$ such 6-cycles if $N_{az}>0$). Write $N=N_{az}$. Define an auxiliary graph $H$ on vertex set $\{x_1,\dots,x_N\}$ such that $x_ix_j$ is an edge if and only if $x_ix_j$ is bad. Observe that $H$ contains no path of length $k-1$. Indeed, if $x_{i_1}x_{i_2}\dots x_{i_{k}}$ is such a path in $H$, then we can choose some vertices $b_1,\dots,b_{k-1}$ in $G$ such that $ax_{i_1}b_1x_{i_2}b_2\dots x_{i_{k-1}}b_{k-1}x_{i_k}a$ is a rainbow cycle of length $2k$, giving a contradiction. It follows that $|E(H)|\leq kN$. So there are at most $kN$ pairs $\{i,j\}$ such that $x_ix_j$ is bad. Similarly, there are at most $kN$ pairs $\{i,j\}$ such that $y_iy_j$ is bad. It follows that if $N\geq 1$ then there are at least $\binom{N}{2}-2kN$ $6$-cycles $ax_iy_izy_jx_ja$ in which each pair of vertices of distance 2 is good.

Write $T=\{(a,z)\not \in W: f(a,z)>(100k)^2+200k\}$. By the argument above, the number of 6-cycles $axyzy'x'a$ in which $c(xy)=c(x'y')$ and each pair of vertices of distance 2 is good is at least
\begin{align*}
\frac{1}{6}\sum_{(a,z)\in T}\left[\binom{N_{az}}{2}-2kN_{az}\right],
\end{align*}
which is at least
\[\sum_{(a,z)\in T}(\alpha f(a,z)^2-\beta f(a,z))\]
for some positive constants $\alpha,\beta$.

On the other hand, if $L$ denotes the number of paths $axyz$ in which $ay,xz$ are both good, then the number of such 6-cycles is at most $100kL$. Indeed, there are $L$ ways to choose $xyzy'$, then $x'$ is uniquely determined by the condition $c(xy)=c(x'y')$, and then there are at most $100k$ possible choices for $a$, since we need $xx'$ to be good. Hence
\[\sum_{(a,z)\in T}(\alpha f(a,z)^2-\beta f(a,z))\leq 100kL.\]

But we have \begin{equation}\label{eq_Lbound} L\leq \sum_{(a,z)\in T}f(a,z)+O(n^2).\end{equation} Indeed, we know that the number of $P_3$s $axyz$ (such that $ay$ and $xz$ are good) having $(a,z) \in W$ is $O(n^2)$, the number of such rainbow $P_3$s $axyz$ with $(a,z)\in T$ is $\sum_{(a,z)\in T}{f(a,z)}$, the number of such rainbow $P_3$s $axyz$ with $(a,z)\not \in T, (a,z)\not \in W$ is at most $((100k)^2+200k)n^2$, and finally, the number of such non-rainbow $P_3$s is at most the number of $P_2$s $xyz$ with $xz$ good, which is $O(n^2)$. It follows that
\[\sum_{(a,z)\in T}(\alpha f(a,z)^2-\beta f(a,z))\leq 100k\sum_{(a,z)\in T}f(a,z)+O(n^2),\]
and hence 
\[\sum_{(a,z)\in T} f(a,z)^2\leq  A\sum_{(a,z)\in T}f(a,z)+Bn^2\]
for some positive constants $A,B>0$. But we have
\begin{align*}\sum_{(a,z)\in T} f(a,z)^2\geq \left[\sum_{(a,z)\in T} f(a,z)\right]^2\cdot\frac{1}{|T|}\geq\left[\sum_{(a,z)\in T} f(a,z)\right]^2\cdot\frac{1}{n^2}.\end{align*}

We get
\[\left[\sum_{(a,z)\in T} f(a,z)\right]^2\leq  An^2\sum_{(a,z)\in T}f(a,z)+Bn^4,\]
which gives $\sum_{(a,z)\in T} f(a,z)=O(n^2)$. The statement of the lemma then follows using \eqref{eq_Lbound}.
\end{proof}

\begin{proof}[Proof of Theorem \ref{theorem_pathcycle}]
For $k\geq 3$, Lemma \ref{lemma_pathsoflength2} shows that there are $O(n^2)$ copies of $P_2$, and Lemmas \ref{lemma_badbadpath}, \ref{lemma_goodbadpath} and \ref{lemma_goodgoodpath} show that there are $O(n^2)$ copies of $P_3$. The required upper bound then follows by repeated application of Lemma \ref{lemma_pathsoflength2}. For the lower bound, take an $(l+1)$-partite graph with vertex classes $X_1, \dots, X_{l+1}$ such that $|X_i|=1$ if $i$ is even and $|X_i|=\Theta(n)$ if $i$ is odd, and join vertices $x$ and $y$ if and only if $x\in X_i$ and $y\in X_j$ with $i-j=\pm 1$. (The edge-colouring is arbitrary.)

When $k=2$, the number of paths of length $2$ is $O(n^2)$ by Lemma \ref{lemma_pathsoflength2}, and the number of paths of length $1$ is at most $\ex^*(n,C_4)=\Theta(n^{3/2})$ (see \cite{keevash2007rainbow}). The required upper bound then follows by repeated application of Lemma \ref{lemma_pathsoflength2}. For the lower bound, we can take a $C_4$-free $d$-regular graph on $\Theta(n)$ vertices with $d=\Theta(n^{1/2})$.
\end{proof}

We now prove Theorem \ref{theorem_rainbowcycle}. Although the upper bound is proved for $k=2$ and the lower bound is proved for $k\geq 3$ in \cite{gerbner2019generalized}, we include proofs of these for completeness.

\begin{proof}[Proof of Theorem \ref{theorem_rainbowcycle}]
	Consider first the case $k=2$. For the upper bound, observe that there can be no bad pair if there is no rainbow $C_4$, thus any two vertices $x$ and $z$ are contained in $O(1)$ 4-cycles of the form $xyzw$. The upper bound $\ex(n,C_4,\textnormal{rainbow-}C_4)=O(n^2)$ follows. For the lower bound when $k=2$, let $A$ be a Sidon set in $\ZZ_n$ of size $\Theta(\sqrt{n})$, i.e., a set such that whenever $a, b, a', b'\in A$ with $a+b=a'+b'$ then $(a,b)=(a',b')$ or $(a,b)=(b',a')$. (See e.g. \cite{erdos1941problem} for the construction of such sets.) Partition $A$ into two subsets $A_1,A_2$ of size $\Theta(\sqrt{n})$ each. Let $G$ be a 4-partite graph with vertex classes $X_{00},X_{01},X_{10},X_{11}$ each being copies of $\ZZ_{n}$, and edges given as follows. If $x_{00}\in X_{00}, x_{01}\in X_{01},x_{10}\in X_{10},x_{11}\in X_{11}$, then we join:
	\begin{itemize}
		\item $x_{00}$ to $x_{10}$ by an edge of colour $a_1$ if $x_{10}-x_{00}=a_1\in A_1$;\vspace{-6pt}
		\item $x_{00}$ to $x_{01}$ by an edge of colour $a_2$ if $x_{01}-x_{00}=a_2\in A_2$;\vspace{-6pt}
		\item $x_{10}$ to $x_{11}$ by an edge of colour $a_2$ if $x_{11}-x_{10}=a_2\in A_2$;\vspace{-6pt}
		\item $x_{01}$ to $x_{11}$ by an edge of colour $a_1$  if $x_{11}-x_{01}=a_1\in A_1$.
	\end{itemize}
It is easy to check that the graph we get is properly edge-coloured with no rainbow $C_4$, has $4n$ vertices, and the number of 4-cycles is $n|A_1||A_2|=\Theta(n^2)$.\medskip

Now consider the lower bound for $k\geq 3$. Take a $(2k)$-partite graph with vertex classes $X_1, \dots, X_{2k}$, where $|X_1|=|X_2|=|X_4|=|X_5|=1$, $|X_6|=|X_8|=|X_{10}|=\dots=|X_{2k}|=n$, $|X_3|=n$ and $|X_7|=|X_9|=\dots=|X_{2k-1}|=1$. Join two vertices $x$ and $y$ by an edge if and only if $x\in X_i$, $y\in X_j$ with $i-j\equiv \pm 1$ mod $2k$. Give the unique edge $X_1$ to $X_2$ and the unique edge $X_4$ to $X_5$ colour 1, and arbitrary distinct colours to the remaining edges. It is easy to see that any $2k$-cycle must contain both of the edges of colour $1$, there are $\Theta(n)$ vertices and $\Theta(n^{k-1})$ copies of $C_{2k}$.\medskip

It remains to prove the upper bound for $k\geq 3$. Given a $2k$-cycle $x_1\dots x_{2k}x_1$, define its pattern to be the list of $i$ such that $x_ix_{i+2}$ is good (indices understood mod $2k$), together with the list of pairs $(i,j)$ such that $c(x_ix_{i+1})=c(x_jx_{j+1})$. Note that there are finitely many patterns, so it suffices to show that for each pattern the number of $2k$-cycles of that pattern is $O(n^{k-1})$.
	
	Consider first the case $k\geq 4$. Assume that we have a pattern and an $i$ such that $x_{i-1}x_{i+1}$ is good but $x_{i-3}x_{i-1}$ is bad in the pattern. Then we can choose vertices $x_{i+1}x_{i+2}\dots x_{i+2k-4}$ in $O(n^{k-2})$ ways, since we have to pick a path of length $2k-5$. (Note that $x_{i+2k-4}=x_{i-4}$.) Then, by Lemmas \ref{lemma_badbadpath} and \ref{lemma_goodbadpath}, there are at most $O(n)$ ways of choosing the path $x_{i-4}x_{i-3}x_{i-2}x_{i-1}$ according to the pattern (since $x_{i-3}x_{i-1}$ has to be bad). Then there are at most $100k$ possible ways of choosing $x_i$, since $x_{i-1}x_{i+1}$ is good. So we get $O(n^{k-1})$ $2k$-cycles for these patterns.
	
	So (when $k\geq 4$) it remains to consider the case when there is no $i$ such that $x_{i-1}x_{i+1}$ is good but $x_{i-3}x_{i-1}$ is bad. Observe that for any $2k$-cycle $x_1\dots x_{2k}x_1$, at least one (in fact, at least two) of the pairs $x_2x_4,x_4x_6,\dots,x_{2k}x_2$ has to be good (otherwise we can find a rainbow $C_{2k}$). So it remains to consider patterns such that each of these pairs is good. Similarly, we may assume that each of $x_1x_3,\dots,x_{2k-1}x_1$ is a good pair.
	
	Now consider the colours for the pattern. We must have a pair of different edges with the same colour. We may assume that we have $c(x_1x_2)=c(x_ix_{i+1})$ for some $i$ with $3\leq i\leq k+1$. Then we can choose $x_2x_3\dots x_{2k-1}$ in $O(n^{k-1})$ ways (since it is a path of length $2k-3$). Then $x_1$ is uniquely determined by the condition $c(x_1x_2)=c(x_ix_{i+1})$, and then there are at most $100k$ possible choices for $x_{2k}$ (according to the pattern), since $x_1x_{2k-1}$ is good. This gives $O(n^{k-1})$ $2k$-cycles of this pattern, as required.\medskip
	
	It remains to consider the case $k=3$. Observe that if $k=3$, then for any edge $ab$ there is at most one way to extend this edge to a path $abc$ such that $ac$ is bad. Indeed, if we have two different extensions $abc$ and $abc'$ then there is a rainbow $6$-cycle of the form $axcbc'x'a$. Consider any pattern, we show that there are $O(n^2)$ $6$-cycles of that pattern. We may assume that $c(x_1x_2)=c(x_ix_{i+1})$ for some $i\in \{3,4\}$. If $x_5x_1$ is good in the pattern, then we are done exactly as above: we can choose $x_2x_3x_4x_5$ in $O(n^2)$ ways, then $x_1$ is determined by the condition $c(x_1x_2)=c(x_ix_{i+1})$, and there are at most $100k$ choices for $x_6$. So we may assume that $x_5x_1$ is bad.\vspace{-3pt}
	\subparagraph{Case 1: $i=4$.} Then the same argument shows that we are done if $x_2x_4$ is good. So we may assume that $x_2x_4$ and $x_5x_1$ are both bad. Then we can choose $x_6x_1x_2x_3$ in $O(n^2)$ ways, and we can extend $x_2x_3$ to a path $x_2x_3x_4$ such that $x_2x_4$ is bad in at most one way, and similarly we can extend $x_1x_6$ in at most one way to get $x_1x_6x_5$. Then all the vertices are determined, so we get $O(n^2)$ copies.
	\subparagraph{Case 2: $i=3$.} 
	There are $O(n^2)$ ways of choosing $x_3x_2x_1x_6$, and then there is at most one way of extending $x_1x_6$ to a path $x_1x_6x_5$ such that $x_1x_5$ is bad, and there is at most one way of picking $x_4$ such that $c(x_3x_4)=c(x_1x_2)$. So we get $O(n^2)$ copies of $C_6$, as required.
\end{proof}

\subsection{Odd cycles}
We now turn to the case of odd cycles. 
Once we have established Theorem \ref{theorem_pathcycle}, the proof of Theorem \ref{theorem_oddcycle} is essentially the same as the proof of Gishboliner and Shapira \cite{gishboliner2018generalized} for the non-rainbow version of the problem.
\begin{proof}[Proof of Theorem \ref{theorem_oddcycle}]
	The lower bounds follow from the fact $\ex(n,F,\textnormal{rainbow-}H)\geq \ex(n,F,H)$ and the corresponding results for the non-rainbow problem, see \cite{gishboliner2018generalized}. (Note that the only difficult case is when $k=2$.)
	
	For the upper bound when $k=2$, observe that there can be no bad pair of vertices if there is no rainbow $C_4$, hence the number of $(2l+1)$-cycles is at most $100k=200$ times the number of paths of length $2l-1$, which is $O(n^{l+1/2})$ by Theorem \ref{theorem_pathcycle}.
	
	Now consider the case $k\geq 3$. Given a path $P: x_1x_2\dots x_{2l-1}$ of length $2l-2$ in $G$, write $X_P=N(x_1)\setminus V(P)$ and $Y_P=N(x_{2l-1})\setminus V(P)$. Then the number of ways of extending path $P$ to a cycle $x_1x_2\dots x_{2l+1}x_1$ is $e(X_P,Y_P)$. But this is at most the number of paths of length $2$ starting at $x_1$ in the graph $G[\{x_1\}\cup X_P\cup Y_P]$, which is $O(1+|X_P|+|Y_P|)$ by Lemma \ref{lemma_pathsoflength2}. It follows that $P$ extends to at most $O(1+|X_P|+|Y_P|)$ cycles of length $2l+1$. But $|X_P|$ is the number of ways of extending $P$ to a path $x_0x_1x_2\dots x_{2l-1}$, and similarly, $|Y_P|$ is the number of ways of extending $P$ to a path $x_1\dots x_{2k}$. It follows that if the number of paths of length $s$ is $p_s$, then 
$	\sum_{P}{|X_P|}=O(p_{2l-1})$,
	and similarly for $Y_P$. Hence the number of cycles of length $2l+1$ is $O(p_{2l-2})+O(p_{2l-1})$, which is $O(n^l)$ by Theorem~\ref{theorem_pathcycle}.
\end{proof}

\section{Forbidden rainbow $C_{2k+1}$}\label{section_oddcycle}

In this section we prove the following result, which is the only non-trivial case of Theorem \ref{theorem_allcycles} with $t$ odd.

\begin{theorem}\label{theorem_oddodd}
	If $k\geq l\geq 2$ are positive integers, then $\ex(n,C_{2l+1},\textnormal{rainbow-}C_{2k+1})=\Theta(n^{2l-1}).$
\end{theorem}

From now on, unless otherwise stated, we will assume that $k\geq l\geq 2$ are integers, $G$ is a properly edge-coloured graph of order $n$ with no rainbow $C_{2k+1}$, and $c$ denotes the edge-colouring. Also, we will say (as before) that a pair $x,y$ of vertices is bad if $|N(x)\cap N(y)|\geq 100k$, and good otherwise.



We will deduce Theorem \ref{theorem_oddodd} from the following two lemmas.

\begin{lemma}\label{lemma_oddwlograinbow}
	Let $G$ be any properly edge-coloured graph, and let $l\geq 2$ be an integer. Then the number of non-rainbow copies of $C_{2l+1}$ in $G$ is $O(n^{2l-1})+O(\textnormal{number of rainbow $C_{2l+1}$s in $G$})$.
\end{lemma}

%

\begin{lemma}\label{lemma_P2odd}
	Let $k\geq l\geq 2$ be integers and let $G$ be a properly edge-coloured graph with no rainbow $C_{2k+1}$. Assume that every edge of $G$ is contained in a rainbow $C_{2l+1}$. Then for every $a\in V(G)$ the number of paths $axy$ of length $2$ starting at $a$ in $G$ is $O(n)$.
\end{lemma}

\begin{proof}[Deducing Theorem \ref{theorem_oddodd}]
	For the lower bound, take a $(2l+1)$-partite graph with vertex classes $X_1, \dots, X_{2l+1}$ all being copies of $\{1,\dots,n\}$. Join any $x\in X_1$ to $x\in X_2$ by an edge of colour 1, and also $x\in X_3$ to $x\in X_4$ by an edge of colour $1$. For all $i\not =1,3$, join each pair of vertices $x,y$ with $x\in X_i$, $y\in X_{i+1}$ by an edge of arbitrary unused colour (with indices understood mod $2l+1$). It is clear that the graph we get is properly edge-coloured, there are $\Theta(n)$ vertices and $\Theta(n^{2l-1})$ copies of $C_{2l+1}$. Furthermore, no copy of $C_{2k+1}$ is rainbow, since any $C_{2k+1}$ must contain an edge between each pair of $X_i, X_{i+1}$ (otherwise it would be a subgraph of a bipartite graph). The lower bound follows.
	
	Now consider the upper bound. By Lemma \ref{lemma_oddwlograinbow}, it suffices to show that if $G$ contains no rainbow $C_{2k+1}$ then the number of rainbow $C_{2l+1}$s is $O(n^{2l-1})$. For this, we may assume that any edge is contained in a rainbow copy of $C_{2l+1}$. But then, by Lemma \ref{lemma_P2odd}, for any vertex $a\in V(G)$ there are $O(n)$ paths of length $2$ starting at $a$. By repeated application of this fact, it follows that for any $a$ there are $O(n^{l})$ paths of length $2l$ starting at $a$, and hence there are $O(n^{l+1})\leq O(n^{2l-1})$ copies of $C_{2l+1}$.
\end{proof}

\begin{proof}[Proof of Lemma \ref{lemma_oddwlograinbow}]
	We will consider patterns of $(2l+1)$-cycles. Recall that the pattern $\mathcal{P}$ of a $(2l+1)$-cycle $x_1\dots x_{2l+1}x_1$ is the list of $i$ such that $x_ix_{i+2}$ is good, together with the list of pairs $(i,j)$ such that $c(x_ix_{i+1})=c(x_jx_{j+1})$ (with the indices understood mod $2l+1$). Since there are finitely many patterns, it suffices to show that for any non-rainbow pattern the required bound holds for cycles of that pattern.
	
	Consider first the case when there are three edges with the same colour in a pattern $\mathcal{P}$, say $x_px_{p+1}$, $x_qx_{q+1}$, $x_rx_{r+1}$. Then we can pick $(x_i)_{i\not =p,q}$ in $O(n^{2l-1})$ ways, and there is at most one way of extending those points to a $(2l+1)$-cycle of the appropriate pattern. This shows that there are $O(n^{2l-1})$ cycles with this pattern.
	
	Now consider the case when there are two different colours such that each of them appears at least twice as the colour of an edge. For both of these colours, pick two edges of the appropriate colour. So we have $c(e)=c(e')$ and $c(f)=c(f')$ in our pattern for four different edges $e,e',f,f'$. Note that we must have $e\cup e'\not =f\cup f'$. So we can pick $i,j$ such that $x_i\in (e\cup e')\setminus (f\cup f')$ and $x_j\in (f\cup f')\setminus (e\cup e')$. Then picking the vertices $(x_a)_{a\not =i,j}$ determines the $(2l+1)$-cycle uniquely by the colour conditions. It follows that there are $O(n^{2l-1})$ cycles of this pattern.
	
	It remains to consider patterns $\mathcal{P}$ in which there is only one pair of edges of the same colour, say $c(x_ix_{i+1})=c(x_jx_{j+1})$, with $i\not =j-1, j, j+1$. Given a choice $X=\{x_a: a\not =i,j\}$ of all vertices except $x_i,x_j$, consider the number of ways of extending $X$ to a $(2l+1)$-cycle. Write $d_1=|N(x_{i-1})\cap N(x_{i+1})\setminus X|$ and $d_2=|N(x_{j-1})\cap N(x_{j+1})\setminus X|$. Then the number of ways of extending $X$ to a $(2l+1)$-cycle of pattern $\mathcal{P}$ is at most $\min\{d_1,d_2\}$, whereas the number of ways of extending $X$ to a rainbow $C_{2l+1}$ is at least $(d_1-5l)(d_2-5l)$. 
	But we have $\min\{d_1,d_2\}\leq 10l+\max\{0,(d_1-5l)(d_2-5l)\}$, so the number of extensions of pattern $\mathcal{P}$ is at most $O(1)$ plus the number of rainbow extensions. Summing over all possible choices of $X$, we get the required bound.
\end{proof}

Lemma \ref{lemma_P2odd} is proved similarly to Lemma \ref{lemma_pathsoflength2}.

\begin{proof}[Proof of Lemma \ref{lemma_P2odd}]
	Given a bipartition $V(G)=X\cup Y$ of the vertex set of $G$, let $G_{X,Y}$ be the corresponding bipartite graph obtained from $G$ (i.e., $G_{X,Y}$ is obtained by deleting all edges inside $X$ and inside $Y$). Since a random bipartition is expected to preserve a quarter of all paths of length $2$ starting at $a$, it suffices to show that for every bipartition $V(G)=X\cup Y$ with $a\in Y$, the number of paths of length $2$ starting at $a$ in $G_{X,Y}$ is $O(n)$, where the implied constant is independent of the bipartition. So let $V(G)=X\cup Y$ be any bipartition. Write $X_1=N_G(a)\cap X$ and $Y_1=N_G(X_1)\cap Y\setminus\{a\}$, so that we would like to show $e_{G_{X,Y}}(X_1,Y_1)=O(n)$. It suffices to show that $G_{X,Y}[X_1\cup Y_1]$ does not contain a $(100k)$-ary tree of depth $2k$.
	
	Suppose it contains such a tree, then it also contains a $(100k)$-ary tree $T$ of depth $2k-1$ rooted at some $x\in X_1$. Since $ax\in E(G)$, the edge $ax$ of $G$ is contained in a rainbow cycle of length $2l+1$ in $G$. Hence we can find a rainbow path $P: az_1z_2\dots z_{2l-1}x$ of length $2l$ from $a$ to $x$ in $G$. Then we can recursively find distinct vertices $x=x_1, x_2, \dots, x_{2(k-l)+1}$ on our tree $T$ such that \begin{itemize}
	\item for all $i$ we have $x_ix_{i+1}\in E(G_{X,Y})$
	\item for all  $i$ even we have $x_i\in Y_1\setminus V(P)$;
	\item for all $i\geq 3$ odd we have $x_i\in X_1\setminus V(P)$;
	\item for all $i$, $c(x_ix_{i+1})$ does not appear on the path $az_1z_2\dots z_{2l-1}x_1\dots x_i$;
	\item the colour $c(ax_{2(k-l)+1})$ does not appear on the path $az_1z_2\dots z_{2l-1}x_1\dots x_{2(k-l)}$.
\end{itemize}
But then $az_1z_2\dots z_{2l-1}x_1x_2\dots x_{2(k-l)+1}a$ is a rainbow cycle of length $2k+1$ in $G$, giving a contradiction.
\end{proof}

\section{Deducing \texorpdfstring{Theorem \ref{theorem_allcycles} and Theorem \ref{proposition_triangle}}{Theorem 1.2 and Theorem 1.6}}

We now summarise how we deduce each case in Theorem \ref{theorem_allcycles}.

\begin{proof}[Proof of Theorem \ref{theorem_allcycles}]
	We have the following cases.
	\begin{itemize}
		\item If $s=t=4$, then the result follows from Theorem \ref{theorem_rainbowcycle}. If $t=4, s\not =4$ and $s$ is even, then it follows from Theorem \ref{theorem_GMMPcycle}. If $t=4$ and $s$ is odd, it follows from Theorem \ref{theorem_oddcycle}.
		\item If $s,t$ are even with $s\not =t$, then the result follows from Theorem \ref{theorem_GMMPcycle}.
		\item If $s=t\geq 6$ is even, then the result follows from Theorem \ref{theorem_rainbowcycle}.
		\item If $t\geq 6$ is even and $s$ is odd, then the result follows from Theorem \ref{theorem_oddcycle}.
		\item If $s, t$ are odd with $s\leq t$, then the result follows from Theorem \ref{theorem_oddodd}.
		\item If $t$ is odd, and $s$ is even or $s>t$, then the upper bound is trivial, and for the lower bound we can take a blowup of $C_{s}$, (i.e., we replace each vertex of $C_s$ by $n$ vertices and each edge by a complete bipartite graph. The edge-colouring is arbitrary.)
	\end{itemize}
\end{proof}

Finally, we prove Theorem \ref{proposition_triangle} concerning triangles.
\begin{proof}[Proof of Theorem \ref{proposition_triangle}]
	For the upper bound $\ex(n,C_3,\textnormal{rainbow-}C_{2k})=O(\ex^*(n,C_{2k}))$, observe that the number of triangles containing a good pair is at most $100k|E(G)|$, since we can pick the good pair in at most $|E(G)|$ ways. So it suffices to show that the number of paths $xyz$ with $xz$ bad is $O(|E(G)|)$. But for any $y\in V(G)$, if we define an auxiliary graph $H_y$ with vertex set $N(y)$ and edges being the bad pairs, then there can be no path $x_1\dots x_{k}$ of length $k-1$ in $H_y$ (otherwise we can find a rainbow cycle $yx_1b_1x_2b_2\dots x_ky$). It follows that $H_y$ has at most $k|V(H_y)|=k\deg_G(y)$ edges, so each $y$ is contained in at most $k\deg(y)$ paths $xyz$ with $xz$ bad. But $\sum_y\deg(y)=2|E(G)|$, giving the required bound.
	
	For the lower bound, the statements $\ex(n,C_3,\textnormal{rainbow-}C_{2k})\geq \ex(n,C_3,C_{2k})$ and $\ex(n,C_3,\textnormal{rainbow-}C_{2k+1})\geq \ex(n,C_3,C_{2k+1})$ are clear, and the lower bounds $\ex(n,C_3,C_{2k})=\Omega(\ex(n,\{C_4,C_6,\dots,C_{2k}\}))$, $\ex(n,C_3,C_{2k+1})=\Omega(\ex(n,\{C_4,C_6,\dots,C_{2k+1}\}))$ follow from Theorem \ref{theorem_trianglenonrainbow}.
	
	Finally, we prove that $\ex(n,C_3,\textnormal{rainbow-}C_{2k})=\Omega(n^{1+1/k})$ when $k$ is odd and $\ex(n,C_3,\textnormal{rainbow-}C_{2k+1})=\Omega(n^{1+1/k})$ when $k$ is even. Take a $B_k$-set $A$ of size $\Theta(n^{1/k})$ in $\ZZ_n$, that is, a set such that any $m\in \ZZ_n$ can be written as $a_1+\dots+a_k$ with $a_i\in A$ in at most one way (ignoring permutations of the summands). (See \cite{bose1960theorems} for the construction of such `dense' $B_k$-sets.) Then we take a tripartite graph $G$ with vertex classes $X_1, X_2, Y$ all being copies of $\ZZ_n$ and edges given as follows. We join $x\in X_1$ to $x\in X_2$ by an edge of colour $0$, and we join $x\in X_i$ to $x+a\in Y$ by an edge of colour $(a,i)$ for $i=1,2$. Clearly, $G$ has $\Theta(n^{1+1/k})$ triangles. We claim that this graph contains no rainbow $C_{2k}$ if $k$ is odd and no rainbow $C_{2k+1}$ if $k$ is even. Indeed, assume that $k$ is odd an there is a rainbow $C_{2k}$. Then it must be of the form $x_1y_1x_2y_2\dots x_ky_k x_1$ with $y_j\in Y$ and $x_i\in X_1\cup X_2$. Then we get a representation $0=a_1-b_1+a_2-b_2+\dots+a_k-b_k$ with $a_i, b_j\in A$ by letting $a_i=y_i-x_i$, $b_i=y_i-x_{i+1}$ (where $x_{k+1}=x_1$). So the $a_i$ must be a permutation of the $b_j$. But $k$ is odd, so we have $|\{x_1,\dots,x_k\}\cap X_1|\not =|\{x_1,\dots,x_k\}\cap X_2|$, and hence there exist $i$ and $j$ such that $a_i=b_j$ and $x_i, x_{j+1}$ are in the same vertex class $X_l$. But then $c(x_iy_i)=c(y_jx_{j+1})$, so the cycle is not rainbow, giving a contradiction. The case when $k$ is even and $G$ contains a rainbow $(2k+1)$-cycle is similar.
\end{proof}

\bibliography{Bibliography}
\bibliographystyle{abbrv}	
	
\end{document}